\documentclass[11pt]{article}
\usepackage{amssymb, latexsym, mathtools, amsthm, amsmath}
\usepackage{txfonts}
\usepackage{authblk}
\begingroup\makeatletter
\@for\theoremstyle:=definition,remark,plain\do{\expandafter\g@addto@macro\csname th@\theoremstyle\endcsname{
\addtolength\thm@preskip\parskip}}
\endgroup

\usepackage[margin=1in]{geometry}
\renewenvironment{abstract}
 { \normalsize
  \list{}{\setlength{\leftmargin}{.0cm}%
    \setlength{\rightmargin}{\leftmargin}}%
  \item {\bf \abstractname.}\relax}
 {\endlist}
\usepackage{titlesec}
 \titleformat{\paragraph}
{\normalfont\normalsize\em\bfseries}{\theparagraph}{1em}{$\blacktriangleright$\hspace{0.2cm}}
\titlespacing*{\paragraph}{0pt}{3.25ex plus 1ex minus .2ex}{0.5ex plus .2ex}

\usepackage{pgf,tikz}
\usetikzlibrary{matrix}
\usetikzlibrary{chains}
\usetikzlibrary{decorations, decorations.pathmorphing}
\usetikzlibrary {shapes}

\definecolor{dnrbl}{rgb}{0,0,0.3}
\definecolor{dnrgr}{rgb}{0,0.3,0}
\definecolor{dnrre}{rgb}{0.5,0,0}
\usepackage[colorlinks=true, citecolor=dnrgr, linkcolor=dnrre, urlcolor=dnrre] {hyperref}
\usepackage{xcolor,colortbl}
\usepackage{enumerate}
\usepackage{pdfsync}
\usepackage[numbers]{natbib}
\usepackage{parskip}
\usepackage{tabularx}
\theoremstyle{plain}
\newtheorem{thm}{Theorem}[section]
\newtheorem{prop}[thm]{Proposition}

\newtheorem{lem}[thm]{Lemma}

\newtheorem{coro}[thm]{Corollary}

\theoremstyle{definition}
\newenvironment{rem}{{\bf Remark}.}{\hfill$\blacksquare$}
\newtheorem{defi}[thm]{Definition}

\newcommand{\Nat}{\mathbb{N}}

\newcommand{\Ws}{\textsf{W}}

\newcommand{\ds}{\textup{\textsf{d}}}

\newcommand{\restr}{\upharpoonright}  

\newcommand{\un}{\uparrow} 
\newcommand{\de}{\downarrow} 

\newcommand{\sqbrad}[2]{\{\hspace{0.03cm}{#1} : {#2}\hspace{0.03cm}\}}

\DeclarePairedDelimiter{\dbra}{\llbracket}{\rrbracket}

\newcommand{\bigo}[1]{\textsf{O}\hspace{0.02cm}\big({#1}\big)}

\newcommand{\mm}{\mathbf{m}}

\newcommand{\CC}{\mathcal{C}}

 \usepackage{tcolorbox}

\newcommand{\wgt}[1]{\mathop{\mathsf{wgt}}\/\left({#1}\right)}
\newcommand{\abs}[1]{|{#1}|}
\newcommand{\asto}{^{\ast}}

\newcommand{\parb}[1]{\big({#1}\big)}

\newcommand{\TT}{\mathcal{T}}

\newcommand{\ml}{Martin-L\"{o}f }

\newcommand{\pz}{$\Pi^0_1$\ }
\newcommand{\pzt}{$\Pi^0_2$\ }

\newcommand{\eg}{e.g.\ }

\newcommand{\ce}{c.e.\ }

\newcommand{\lce}{left-c.e.\ }

\newcommand{\pf}{prefix-free }

\newcommand{\twome}{2^{\omega}}

\newcommand{\zj}{\emptyset'}
\newcommand{\twomel}{2^{<\omega}}

\newcommand{\wedga}{\ \wedge\ \ }

\newcommand{\PA}{\mathsf{PA}}

\newcommand{\leqT}{\leq_T}
\newcommand{\geqT}{\geq_T}
\newcommand{\equivT}{\equiv_T}

\newcommand{\impl}{\ \Rightarrow\ }

\newcommand{\hthree}{\hspace{0.3cm}}

\newcommand{\Ts}{\mathsf{T}}

\newcommand{\Phim}{\Phi^{-1}}

\newcommand{\zjj}{\emptyset''}

\newcommand{\geqp}{\overset{\textrm{\tiny $+$}}{\geq}}
\newcommand{\leqp}{\overset{\textrm{\tiny $+$}}{\leq}}
\newcommand{\leqt}{\overset{\textrm{\tiny $\times$}}{\leq}}

\newcommand{\eqp}{\overset{\textrm{\tiny $+$}}{=}}

\newcommand{\Mb}{\mathbf{M}}

\newcommand{\sz}{$\Sigma^0_1$\ }
\newcommand{\szn}{$\Sigma^0_1$}

\title{Dimensionality and randomness\thanks{Authors appear in alphabetical order.  Supported by Beijing Natural Science Foundation (IS24013).}}
\author{George Barmpalias}\author{Xiaoyan Zhang} 
\affil{State Key Lab of Computer Science, Institute of Software\\ Chinese Academy of Sciences, Beijing, China}

\begin{document}
\maketitle
\begin{abstract}
Arranging the bits of a random string or real into $k$ columns of a  
two-dimensional array or higher dimensional structure is typically 
accompanied with loss in the Kolmogorov complexity of the columns, which depends on $k$. 
We quantify and characterize this phenomenon for arrays and trees 
and its relationship to negligible classes.
\end{abstract}

\section{Introduction}
The initial segments of an algorithmically random real $x$ in the  sense of \citet{MR0223179}  are incompressible: the length of the shortest 
self-delimiting program generating the $n$-bit prefix $x\restr_n$ of it, its {\em \pf Kolmogorov complexity} $K(x\restr_n)$,  
is at least $n$. On the other hand, the binary expansions of random reals $x$ have  arbitrarily long highly compressible 
{\em segments} (e.g.\ blocks of 0s), although there are infinite binary sequences and matrices 
which only contain highly complex segments \cite{Miller2012, DurandLevinShen08}.
Similarly, any tree of unbounded width computable by a sufficiently random real has a  compressible path \cite{bslBienvenuP16,treeout,dimtree}. 
So a random real may be compressible under different effective arrangements of its bits, as Figure \ref{QlwGmyUgT} illustrates.

This contrasts the fact that effectively extracted strings of bits from a random real $x$ 
are incompressible, and indicates that the compressibility (randomness deficiency) 
 is a function of the number of strings extracted from $x$. 
Our goal is to determine the
\begin{itemize}
\item  limits of incompressibility  in arrays/trees  computable from sufficiently random reals
\item randomness deficiency of such arrays/trees, as a function of their growth in size    
\item class of random reals that do not have the above limitations.
\end{itemize}
The above phenomenon can be interpreted in terms of probabilistic algorithms
and {\em negligible}  classes, immune to
effective probabilistic generation \cite{neglivyuginold}.
Although incompressible strings/reals can be produced by probabilistic machines
(with high probability), nontrivial arrays/trees of incompressible reals cannot. 

\begin{figure}
\scalebox{0.7}{\begin{tikzpicture}[start chain=1 going right, start chain=3 going above, start chain=2 going above,,node distance=0.2mm,
kno/.style={fill=white!60!black}, kna/.style={fill=black!90!gray}, kni/.style={fill=black!80!gray}]
\foreach \y in {1,...,15} \node [kni,on chain=2] {};\hspace{0.2cm}\draw[->,   in = 180, out=0, very thick, gray!60!black]   (0.2,2) to (0.9,2);\hspace{0.2cm}
\node [kna,on chain=3] at (1.2,1.3) {}; \foreach \x in {1,...,3}  \node [kna,on chain=3] {};\node [kno, on chain=3] at (1.5,1.15) {}; \foreach \x in {1,...,4}  \node [kno,on chain=3] {};
\node [kna,on chain=3] at (1.8,1.15) {}; \foreach \x in {1,...,2}  \node [kna,on chain=3] {}; \node [kno,on chain=3] at (2.1,1.15) {}; \foreach \x in {1,...,3}  \node [kno,on chain=3] {};
\end{tikzpicture}\centering\hspace{3cm}
\begin{tikzpicture}[start chain=1 going above, start chain=2 going above,  node distance=0.2mm, kno/.style={fill=white!60!black}, kna/.style={fill=black!90!gray}]
\node [kno,on chain=1] at (1.2,2.15) {}; \foreach \x in {1,...,13}  \node [kno,on chain=1] {}; \node [kna,on chain=1] at (1.5,2) {}; \foreach \x in {1,...,12}  \node [kna,on chain=1] {};
\node [kno,on chain=1] at (1.8,2) {};\foreach \x in {1,...,13}  \node [kno,on chain=1] {}; \node [kna,on chain=1] at (2.1,2) {}; \foreach \x in {1,...,12}  \node [kna,on chain=1] {};
\hspace{0.4cm} \draw[->,   in = 180, out=0, very thick, gray!60!black]   (2.1,4) to (2.9,4);
\hspace{0.4cm} \node [kno,on chain=2] at (3.1,3.55) {}; \foreach \x in {1,...,3}  \node [kno,on chain=2] {};
\node [kna,on chain=2] at (3.4,3.4) {}; \foreach \x in {1,...,2}  \node [kna,on chain=2] {}; \node [kno,on chain=2] at (3.7,3.4) {}; \foreach \x in {1,...,4}  \node [kno,on chain=2] {};
\node [kna,on chain=2] at (4,3.4) {}; \foreach \x in {1,...,3}  \node [kna,on chain=2] {}; \node [kno,on chain=2] at (4.3,3.4) {}; \foreach \x in {1,...,2}  \node [kno,on chain=2] {};
\node [kna,on chain=2] at (4.6,3.4) {}; \foreach \x in {1,...,3}  \node [kna,on chain=2] {}; \node [kno,on chain=2] at (4.9,3.4) {}; \foreach \x in {1,...,4}  \node [kno,on chain=2] {};
\node [kna,on chain=2] at (5.2,3.4) {}; \foreach \x in {1,...,3}  \node [kna,on chain=2] {}; \node [kno,on chain=2] at (5.5,3.4) {}; \foreach \x in {1,...,5}  \node [kno,on chain=2] {};
\node [kna,on chain=2] at (5.8,3.4) {}; \foreach \x in {1,...,3}  \node [kna,on chain=2] {}; \node [kno,on chain=2] at (6.1,3.4) {}; \foreach \x in {1,...,4}  \node [kno,on chain=2] {};
\node [kna,on chain=2] at (6.4,3.4) {}; \foreach \x in {1,...,3}  \node [kna,on chain=2] {};\end{tikzpicture}}\centering\hspace{2cm} 
\caption{Loss of complexity when arranging a real into columns: black and gray columns indicate segments of high and low complexity, respectively.}\label{QlwGmyUgT}\centering
\end{figure}

\subsection{Our contribution}\label{yP4cTWTJZG}

Let $\twome, \twomel, 2^n$ denote the sets of infinite binary sequences, finite binary strings, and $n$-bit binary strings respectively, and let $2^{\leq n}=\bigcup_{i\leq n}2^i$.

The {\em deficiency} of a string $\sigma$ is $\ds(\sigma):=|\sigma|-K(\sigma)$, namely the number of bits between its length and its shortest prefix-free description. Given a set $E$ of strings
\begin{itemize}
\item the {\em deficiency} $\ds(E)$  is the supremum of the deficiencies of the members of $E$
\item $E$ is a  {\em set of incompressible strings} if it has finite deficiency.
\end{itemize}
Any random real can effectively generate an infinite set of incompressible strings, \eg the set of its prefixes.
Toward a stronger statement, we calibrate the size of sets via {\em orders}: nondecreasing unbounded
$g:\Nat\to\Nat$.  We say that $D$ is {\em $g$-fat} if $\max_{i\leq n} \abs{D\cap 2^i}\geq g(n)$, where
$2^i$ denotes the set of $i$-bit strings.
Let $\leqT$ denote the Turing reducibility.
\begin{thm}\label{ODDbi7Cm2s}
Every random real computes an $n/(\log n)^2$-fat set of incompressible strings.
\end{thm}
This fails for considerably fat sets of strings:
\begin{thm}\label{24beca61}
If $g$ is a computable order with $\lim_n n/g(n)=0$, and a random real $z$ computes a $g$-fat set of incompressible strings, then $z\geqT\zj$.
\end{thm}

Let $\mu$ denote the uniform measure on $\twome$.
A set $T$ of strings is a {\em tree} if it is downward-closed with respect to the prefix relation $\preceq$. 
The $\preceq$-maximal $\sigma\in T$ are called {\em deadends}. 
A tree $T$ is 
\begin{itemize}
\item {\em pruned} if it does not have deadends,
\item {\em proper} if its {\em width} $\abs{T\cap 2^{n}}$  is unbounded
\item {\em incompressible} if $\ds(T)<\infty$ and {\em $c$-incompressible} if $\ds(T)< c$, where $c\in\Nat$,
\item {\em weakly-incompressible} if $\liminf_n \ds(T\cap 2^n)<\infty$,
\item {\em perfect} if each string in $T$ has at least two $\preceq$-incomparable strict extensions in $T$. 
\end{itemize}
Note that a pruned tree is proper iff it has infinitely many paths.

Let $[T]$ denote the set of {\em paths through $T$}, namely the reals with all its prefixes in $T$. 
We say that $T$ is  {\em positive} if  $\mu([T])>0$.
All trees we consider are pruned and proper, with the exception of \S\ref{VW3LYXitIA} which concerns trees with deadends. 

\begin{thm}\label{7i3g6ty9SJ}
If a random real $z$ computes a proper pruned incompressible tree, then $z\geqT\zj$.
\end{thm}
This extends the special case of
trees with a computable nondecreasing unbounded lower bound on their width due to \cite{bslBienvenuP16}, and
of perfect trees  under stronger randomness assumptions established in \cite{deniscarlsch21,treeout}.
In contrast:

\begin{thm}\label{YRH2XAwqKL}
Every random real computes a perfect weakly-incompressible tree.
\end{thm}

In \S\ref{d6GV7C9Wxm} we refine our results by replacing {\em incompressibility} with  {\em deficiency growth}.
We have seen that the deficiency $\ds(T\cap 2^n)$ of a tree $T$ 
that is computable by a sufficiently random real has to grow as its width $\abs{T\cap 2^n}$ grows.
Our goal is to determine the necessary deficiency growth given $\abs{T\cap 2^n}$.
An upper bound is given by:

\begin{thm}\label{O65ckpRLDvabb1}
Let $g$ be a computable order with $\forall n,\ g(n+1)\leq g(n)+1$.
Every random real computes a perfect tree $T$  
with $\log |T\cap 2^{n}|=g(n)$ and
\[
\lim_n\parb{\log |T\cap 2^n|- \ds(T\cap 2^n)}=\infty.
\]
\end{thm}

We obtain a general lower bound, of which a representative case is:

\begin{thm}\label{O65ckpRLDvabb}
If a random real $z$ computes a tree $T$ with
\begin{equation}\label{4f229PSV42}
\log \abs{T\cap 2^n}\geq 2\log\log \abs{T\cap 2^n}+ \ds(T\cap 2^n),
\end{equation}
then $z\geqT\zj$. 
\end{thm}

Note that $2\log n$ is a standard computable upper estimate for $K(n)$, and
\[
\ds(\log \abs{T\cap 2^n})=\log \abs{T\cap 2^n}-K(\log \abs{T\cap 2^n})
\]
so Theorem \ref{O65ckpRLDvabb} suggests the possibility that 
$\ds(T\cap 2^n)\leq \ds(\log \abs{T\cap 2^n})$ can replace
\eqref{4f229PSV42} in the statement. 
This is not known and is left as an open question.
It suggests however a direct relation between the  necessary deficiency of $T$
and the deficiency of its size, at each level $n$.

In the case of computable width, a tighter statement is possible.

\begin{thm}\label{O65ckpRLDvabb3}
If a random real $z$ computes a tree $T$ and an order $g$ such that
 $\log \abs{T\cap 2^n}\geq g(n)+\ds(T\cap 2^n)$ and $(\abs{T\cap 2^n})$ is computable, then $z\geqT\zj$.
\end{thm}

A similar statement is shown if we replace the computability of 
$(\abs{T\cap 2^n})$  with that of $g$.

Probabilistically, Theorems \ref{ODDbi7Cm2s} to \ref{YRH2XAwqKL} show that while sets of incompressible strings or reals can be effectively generated 
from any typical oracle, there is a  limitation on the fatness/width beyond which this becomes a 
{\em negligible property}. Negligibility is  associated with  structured information and  depth
\cite{Bennett1988} as in the set of halting programs or 
complete consistent extensions of formal arithmetic \cite{Jockusch197201,MR2258713frank,Levin2013FI}. 

\citet{bslBienvenuP16,BienPortDeep24} exhibited a wider array of such classes by extending the notion of depth to \pz classes, namely 
effectively closed sets, and showed that incomplete random reals do not compute members of {\em deep \pz classes}.
In \S\ref{vbNyewP8tQ} we show
that our results fall outside the reach of the  methodology of \cite{bslBienvenuP16,BienPortDeep24}.
For example, the widths $\abs{T\cap 2^n}$ of the members $T$ of 
a \pz class of proper pruned trees must be lower-bounded by a computable order.
Without generalizing  depth outside \pz classes we demonstrate
that a key property of members of deep \pz classes can be found on trees that are not members of any deep \pz class.  

\begin{thm}\label{FJCm7I5aYa}
There exists a perfect pruned incompressible tree $T\leqT\zjj$ which is not a member of any deep \pz class.
\end{thm}

The classes in our results are topologically large compared to the
deep \pz classes of \cite{bslBienvenuP16,BienPortDeep24}. For example, with respect to the standard topology (see \S\ref{vbNyewP8tQ}):

\begin{thm}\label{41rMOJConYa}
For each $c$, the class of $c$-incompressible perfect trees which are not members of 
any deep \pz class is comeager in the space of $c$-incompressible trees.
\end{thm}
Our examples can also be found inside $\Pi^0_2$ classes.
Similarly, we show that some results of \cite{bslBienvenuP16} about negligible \pz classes hold more 
generally for negligible $\Pi^0_2$ classes. 

\subsection{Relevant literature}\label{HwNnMSaQnJ}
The computational properties of sets of incompressible strings and their relationship with random reals 
are complex. 
This topic has recently gained focus, especially in relation to models of second order arithmetic 
\cite{deniscarlsch21,treeout,pamiller}.
We  review  several known facts. 

One way to obtain a perfect incompressible tree from any random real is to allow arbitrary bits on an infinite set of chosen positions of the real. This was studied in \cite{Indiffexn101} but the required set of positions is noneffective.  Alternatively one can obtain a perfect incompressible tree from two random reals $x,y$, by mixing the segments of two random reals with respect to a computable partition \cite[Lemma 2.6]{BLNg08}.  Although for each random $x$ there exists $y$ such that the resulting tree is incompressible, the class of $(x,y)$ with this property is null.

Since incompressibility is a \pz property, reals that encode complete extensions of Peano arithmetic ($\PA$ reals)  compute incompressible reals, 
trees and sets of strings \cite[\S 4]{KjosHanssenMStrans}. 
However there are perfect pruned incompressible trees which do not compute $\PA$ reals  
\cite{treeout, pamiller}. 

Incompressible trees with deadends can be computationally even weaker:
\begin{thm}[Essentially \citet{luliumajma}]\label{t6A4XwWbKX}
There exists a positive incompressible tree (with deadends) which does not compute any random real.
\end{thm}
We derive this striking fact in \S\ref{VW3LYXitIA} based on \citet{luliumajma}.
The pruned positive incompressible trees can also be separated from the perfect pruned incompressible trees, in terms of computational power \cite{treeout}: 
there exists a pruned perfect incompressible tree $T$ such that no $T$-\ce set of incompressible strings $D$ has $2^n=\bigo{\abs{D\cap 2^n}}$.
The work on deep classes \cite{bslBienvenuP16} applies to the special case of 
{\em effectively-proper} trees $T$, where the width $\abs{T\cap 2^n}$  is lower-bounded by a computable order:
incomplete randoms do not compute any effectively-proper incompressible tree. 
By \cite{dimtree} there exists a pruned (effectively) proper incompressible tree which does not
compute any perfect incompressible tree with computable oracle-use.

\section{Incompressible arrays}

We prove positive and negative results on the fatness of sets that can be computed by incomplete randoms, which are Theorems \ref{ODDbi7Cm2s} and \ref{24beca61} from \S\ref{yP4cTWTJZG}.

Let $\leqt$ denote inequality up to a positive multiplicative constant. 

Recall that $D\subseteq\twomel$ is 
{\em $g$-fat} if $\max_{i\leq n} \abs{D\cap 2^i}\geq g(n)$.
We prove Theorem \ref{ODDbi7Cm2s}.

\begin{thm}\label{9NoMkmNVvP}
Every random real computes a $n/(\log n)^2$-fat incompressible subset of $\twomel$.
\end{thm}\begin{proof}
We give a probabilistic computable construction of a 
$n/(\log n)^2$-fat set $F\subseteq \twomel$ and show that it is incompressible with probability 1.
For each $n$ we randomly choose:
\begin{itemize}
\item $\ell_n\in [2^n, 2^{n+1})$ so $\log\ell_n=n$
\item $F_n\subseteq 2^{\ell_n}$ of size $p_n:=2^n/n^2$ and let $F\cap 2^{\ell_n}:=F_n$.
\end{itemize}
Then $\abs{F\cap 2^{\ell_n}}=p_n$ and $\max_{i\leq 2^n} \abs{F\cap 2^i}=p_n$
so $F$ is $n/(\log n)^2$-fat. 

The construction is a total Turing functional $\Phi$ 
where each oracle $z$ outputs an array $\Phi(z;n)=F_n(z)$ and corresponds to a unique choice
$(\ell_n(z)), (F_n(z))$ of $(\ell_n), (F_n)$. 

Let $L_n:=\sqbrad{z}{K(\ell_n(z))<n}$. By the counting theorem 
\[
\abs{\sqbrad{\ell\in [2^n, 2^{n+1})}{K(\ell)<n}}\leqt 2^{n-K(n)}
\hspace{0.3cm}\textrm{and}\hspace{0.3cm}
\abs{\sqbrad{\sigma\in 2^{\ell}}{K(\sigma)\leq \ell}}\leq 2^{\ell-K(\ell)}.
\]
So the probability of $K(\ell_n)<n$ is $\leqt 2^{-K(n)}$ and  $\mu(L_n)\leqt 2^{-K(n)}$.

Assuming $K(\ell_n)\geq n$,  the probability that we randomly pick one $\sigma\in 2^{\ell_n}$ with $K(\sigma)<\ell_n$ 
is $\leq 2^{-n}$. So while we randomly pick $p_n$ strings from $2^{\ell_n}$ without replacement,
the probability that a given pick $\sigma$ has $K(\sigma)<\ell_n$ remains bounded by
$2^{\ell_n-n}/(2^{\ell_n}-p_n)<2^{1-n}$.

So the probability that at least one of the $p_n$ strings we picked from $2^{\ell_n}$
has $K(\sigma)<\ell_n$ is 
\[
\leq p_n\cdot 2^{1-n}=2/n^2.
\]
Let $B_n:=\sqbrad{z}{\exists \sigma\in F_n(z), K(\sigma)< \ell_n(z)}$
and note that $(L_n\cup B_n)$ is \sz uniformly in $n$.

 Since $\forall z\in \twome-L_n,\ K(\ell_n(z))\geq n$
 we get $\mu(B_n-L_n)\leq 2/n^2$ so
\[
\mu(L_n\cup B_n)=\mu(L_n)+\mu(B_n-L_n) \leqt 2^{-K(n)}+2/n^2
\]
and $\sum_n \mu(L_n\cup B_n)<\infty$.  So $(L_n\cup B_n)$  is a Solovay test.
If $z$ is random, $z\not\in L_n\cup B_n$ for almost all $n$ and 
$F(z):=\bigcup_n (F_n(z))$ is an $n/(\log n)^2$-fat set incompressible strings.
\end{proof}
Towards Theorem \ref{24beca61} from \S\ref{yP4cTWTJZG}, we first review the notion of difference randomness from \citep{FrNgDiff} (also see \cite[Definition 2.1]{DenjoyBHMN14}). 
\begin{defi}
Let $P$ be a \pz class and $U_i$ be a uniform sequence of $\Sigma^0_1$ classes. 
We say that
\begin{itemize}
\item $(P\cap U_i)$ is a {\em difference test} if $\mu(P\cap U_i)\leq 2^{-i}$
\item $(P\cap U_i)$ is a {\em difference Solovay test} if $\sum_i \mu(P\cap U_i)<\infty$
\item $z$ {\em avoids} $(P\cap U_i)$ if $z\not\in P\cap U_i$ for infinitely many $i$
\item $z$ {\em strongly avoids} $(P\cap U_i)$ if $z\not\in P\cap U_i$ for all but finitely many $i$
\item $z$ is {\em difference-random} if it avoids all difference tests. 
\end{itemize}
\end{defi}
\citet{FrNgDiff} showed that
\begin{equation}\label{oDyfIYIan}
\textrm{if $z$ is random then it is  difference-random iff  $z\not\geq_T\zj$.}
\end{equation}
Difference randomness is also characterized by difference Solovay tests:
\begin{lem}\label{xtR47lHczE}
The following are equivalent for  real $z$:
\begin{enumerate}[(i)]
\item $z$ is difference-random
\item $z$ avoids all  all difference Solovay tests
\item $z$ strongly avoids all difference Solovay tests.
\end{enumerate}
\end{lem}\begin{proof}
Since difference tests are also Solovay tests
we get (iii)$\to$(ii)$\to$(i). 

For $\neg$(iii)$\to\neg$ (i) suppose that
 $(P\cap U_i)$ is a difference Solovay test and $z\in P\cap U_i$ for infinitely many $i$ and let
 consider the uniform sequence of $\Sigma^0_1$ classes
 \[
 V_i:=\sqbrad{x}{\abs{\sqbrad{j}{x\in U_j}}>2^{i+1}}
 \]
 and note that $z\in\cap_i \parb{P\cap V_i}$.
Since $\sum_i\mu(P\cap U_i)<\infty$ we have $\exists c\ \forall i,\ \mu(P\cap V_{i+c})\leq 2^{-i}$. 

So $(P\cap V_{i+c})$ is a difference test and $z$ is not difference random.
\end{proof}

We then prove a Lemma, which will be also useful in subsequent sections. 
When $\Phi$ is a Turing functional, for simplicity of notions we do not assume that writing $\ds(\Phi(x;n))\leq k$ implies $\Phi(x;n)\downarrow$. Instead if $\Phi(x;n)\uparrow$ then $\ds(\Phi(x;n))\leq k$ automatically holds for all $k$. 
\begin{lem}\label{DP8sKcDJCf}
Let $\Phi$ be a Turing functional with $\Phi(x;n)\subseteq 2^n$ and let
\begin{align*}
P^k &:= \sqbrad{x}{\ds(\Phi(x;n))\leq k\text{ for all }n}, \\
P^k_{\sigma} &:= \sqbrad{x\in P^k}{\sigma\in \Phi(x;|\sigma|)}.
\end{align*}
Then $\mu(P^k_{\sigma})\leqt 2^{K(k)+k-|\sigma|}$.
\end{lem}\begin{proof}
Let $P^k(s)$ and $P^k_{\sigma}(s)$ be the approximation of $P_k$ and $P^k_\sigma$ at stage $s$, that is, 
\[
P^k(s) := \sqbrad{x}{\forall n\leq s,\ \ds_s(\Phi_s(x;n))\leq k},
\]
\[
P^k_{\sigma}(s)\ := \sqbrad{x\in P^k(s)}{\sigma\in \Phi_s(x;|\sigma|)}.
\]
We define a construction that runs in stages and involves actions 
\begin{equation}\label{DdEzESnqea}
\textrm{$\ds_{s+1}(\sigma)\geq k+1$, namely $K_{s+1}(\sigma)\leq |\sigma|-k-1$}
\end{equation}
at stage $s+1$ where $\ds_{s}(\sigma)\leq k$. Since we do not have complete control of the 
universal \pf machine $U$ underlying $K$, we achieve \eqref{DdEzESnqea} indirectly via a \pf machine $M$
 that we build and is simulated by $U$ via some constant $c$. Under a standard codification of all \pf machines, we get
$K\leq K_M+c$
and by the recursion theorem, without loss of generality, we may use $c$ in the definition of $M$.
To achieve \eqref{DdEzESnqea}  it suffices to 
\begin{itemize}
\item set $K_{M_{s+1}}(\sigma)= |\sigma|-k-1-c$
\item speed-up the enumeration of $U,K$ so that $K_{s+1}\leq K_{M_{s+1}}+c$.
\end{itemize}
So $K_{s+1}$ refers to the stages in our construction and, assuming $K\leq K_M+c$:
\[
K_{M_{s+1}}(\sigma)= |\sigma|-k-1-c\impl 
K_{s+1}(\sigma)\leq |\sigma|-k-1\impl \eqref{DdEzESnqea}.
\]
In this way we give a simple description of the construction where action \eqref{DdEzESnqea} at $s+1$:
\begin{itemize}
\item indicates the $M$-compression $K_{M_{s+1}}(\sigma)= |\sigma|-k-1-c$
\item requires an extra $M$-description of length $|\sigma|-k-1-c$
\item increases the weight of $M$ by $2^{k+1+c-|\sigma|}$ which is the {\em cost} of \eqref{DdEzESnqea}.
\end{itemize}
The existence of this  $M$  follows by the Kraft-Chaitin-Levin theorem as long as 
the total cost of actions \eqref{DdEzESnqea} is at most 1.
With these standard conventions, we state the construction. 

{\bf Construction.} At stage $s+1$ 
\begin{itemize}
\item set $\ds_{s+1}(\sigma)\geq k+1$ for the least $\sigma\in 2^{\leq s}, k\leq s$ with 
$\mu(P^k_{\sigma}(s))> 2^{K_s(k)+k+1+c-|\sigma|}$ 
\item label action $\ds_{s+1}(\sigma)\geq k+1$ as $(s,\sigma)_k$.
\end{itemize}
If such $\sigma,k$ do not exist, go to the next stage.

{\bf Verification.} 
At each stage $s+1$ where action $(s,\sigma)_k$ occurs:
\begin{enumerate}[(i)]
\item $\ds_{s}(\sigma)\leq k$, because otherwise $P^k_{\sigma}(s)$ would be empty
\item  $P^k_{\sigma}(s)$ is  permanently removed from $P^k$: 
$(s,\sigma)_k\neq (t,\tau)_k\impl P^k_{\sigma}(s)\cap P^k_{\tau}(t)=\emptyset$
\item the compression taking place costs $2^{k+1+c-|\sigma|}<2^{-K(k)}\cdot \mu(P^k_{\sigma}(s))$.
\end{enumerate}
Let $D_k$ be the set of pairs $(s,\sigma)$ over all actions  $(s,\sigma)_k$.

By (ii) and (iii) the compression cost due to actions in $D_k$ is at most
\[
\sum_{(s,\sigma)\in D_k} 2^{-K(k)}\cdot  \mu(P^k_{\sigma}(s)) 
\leq 2^{-K(k)}\cdot \sum_{(s,\sigma)\in D_k} \mu(P^k_{\sigma})\leq 2^{-K(k)}.
\] 
So the total compression cost is at most
$\sum_k 2^{-K(k)}<1$ and the construction is justified: 
\begin{itemize}
\item the weight of machine $M$ that we implicitly construct is $\leq 1$
\item by the Kraft-Chaitin-Levin theorem $M$ is well-defined 
\end{itemize}
and by the first clause in the construction, $\forall k,\sigma :\ \mu(P^k_{\sigma})\leq 2^{k+K(k)+1+c-|\sigma|}$. 
\end{proof}
Let $\leqp$ denote inequality up to an additive constant. Theorem \ref{24beca61} now follows from \eqref{oDyfIYIan} and: 
\begin{thm}
If $g$ is a computable order with $\lim_n n/g(n)=0$ then no difference-random computes
a $g$-fat set of incompressible strings.
\end{thm}\begin{proof}
Let $g$ be as in the statement and let $z,\Phi, c$ be such that
\begin{itemize}
\item $\Phi$ be a Turing functional with $\Phi(z;i)\subseteq 2^i$
\item $\Phi(z):=\bigcup_i \Phi(z;i)$ is a $g$-fat set with $\ds(\Phi(z))<c$.
\end{itemize}
 Let $P\ :=\ \sqbrad{x}{\forall i,\ \parb{\ds(\Phi(x;i))\leq c\ \vee\  \Phi(x;i)\un}}$ and
 \[
 P_{\sigma}\ :=\ \sqbrad{x\in P}{\sigma\in \Phi(x;|\sigma|)}
 \hspace{0.3cm}\textrm{and}\hspace{0.3cm}
Q_n^i:=\sqbrad{z}{\abs{\Phi(z;i)}\geq g(n)}
 \]
so by Lemma \ref{DP8sKcDJCf} we have   $\mu(P_{\sigma})\leqt 2^{-|\sigma|}$ and $\forall i,\ \sum_{\sigma\in 2^i} \mu(P_{\sigma})\leqt 1$.

Since each $x\in P\cap Q_n^i$ belongs to $\geq g(n)$ of the $P_{\sigma}, \sigma\in 2^i$, 
\[
\sum_{\sigma\in 2^i} \mu(P_{\sigma}\cap Q_n^i)\geq g(n)\cdot \mu(P\cap Q_n^i)
\]
and $\mu(P\cap Q_n^i)\leqt 1/g(n)$. Let $Q_n:=\cup_{i\leq n} Q_n^i$ so 
\[
\mu(P\cap Q_n)\leq \sum_{i\leq n} \mu(P\cap Q_n^i)\leqt n/g(n).
\]
By the hypothesis there exists a computable $f$ such that $\mu(P\cap Q_{f(n)})\leq 2^{-n}$ so
$(P\cap Q_{f(n)})$ is a difference test. Also 
$z\in P\cap Q_n$ for each $n$ so $z$ is not difference random.
\end{proof}

\section{Incompressible and weakly-incompressible trees}

In this section we investigate whether incompressible and weakly-incompressible trees can be computed by incomplete randoms
(see Theorems \ref{7i3g6ty9SJ} and \ref{YRH2XAwqKL} from \S\ref{yP4cTWTJZG}). Here we assume that the output of each Turing functional $\Phi$ is a pruned tree; in particular the output of $\Phi(x;n)$, if defined, is a pruned tree whose leaves have length $n$. 

The proof of Theorem \ref{7i3g6ty9SJ} relies on the following consequence of Lemma \ref{DP8sKcDJCf}.

\begin{lem}\label{6xHgSw7WOv}
Let $k_0$ be a constant, $\Phi$ be a Turing functional and
\begin{align*}
Q &:= \sqbrad{x}{\ds(\Phi(x;n))\leq k_0\text{ for all }n}, \\
Q_n^m &:= \sqbrad{x\in Q}{\log |\Phi(x;n)|\geq m}.
\end{align*}
Then $\mu(Q_n^m)\leqt 2^{-m}$ where the implicit constant depends on $k_0$.
\end{lem}\begin{proof}
Let $Q_\sigma:=\sqbrad{x\in Q}{\sigma\in \Phi(x;|\sigma|)}$. By Lemma \ref{DP8sKcDJCf} there exists a constant $c$ with $\mu(Q_\sigma)< 2^{c-|\sigma|}$.
So for each $n$:
\[
\sum_{\sigma\in 2^n}  \mu(Q_\sigma)< 2^{c}.
\]
Fix $m$. Each $x\in Q_n^m$ is in $\geq 2^{m}$ of the $Q_\sigma$ in the sum. So $\mu(Q_n^m)<2^{c-m}$.
\end{proof}
Theorem \ref{7i3g6ty9SJ} follows from \eqref{oDyfIYIan} and: 

\begin{thm}\label{NtI4AEpWrs}
If $z$ is random and computes a proper pruned incompressible tree, then $z$ is not difference-random.
\end{thm}\begin{proof}
Assume that $z$ satisfies the hypothesis, so there exists a proper tree $T$, constant $k_0$ with $\ds(T)\leq k_0$, and
a tree-functional $\Phi$ such that 
$\Phi(z)=T$. Let 
\[
g(m):=\min\sqbrad{n}{\abs{T\cap 2^n}\geq 2^m}
\]
so $g\leq_T z$. Note that
$E_n^m:=\sqbrad{x}{|\Phi(x;n)|\geq 2^m}$ is $\Sigma^0_1$, uniformly in $m,n$ and 
\[
P:=\sqbrad{x}{\forall n,\ \ds(\Phi(x;n))\leq k_0}\in\Pi^0_1.
\]
%
%
Since $\abs{T\cap 2^{g(m)}}\geq 2^m$, $\ds(T)\leq k_0$ and $T$ is pruned:
$n\geq g(m)\impl z\in E_n^m\cap P$.

Let $Q_n^m:=E_n^m\cap P$ so by Lemma \ref{6xHgSw7WOv} we have $\mu(Q_n^m)\leq 2^{c-m}$, where $c$ is a constant.

Let $(m_s)$ be a computable enumeration of $\emptyset'$ without repetitions, so 
$(Q_s^{m_s})$ is a Solovay difference-test.
If $g(m_s)> s$ for almost all $s$ then $z\geq_T g\geq_T\zj$.
Otherwise for infinitely many $s$ we have 
 $g(m_s)\leq s$ and so $z\in Q_s^{m_s}$.
By Lemma \ref{xtR47lHczE},  $z$ is not difference random.
\end{proof}

The combination of Theorem \ref{NtI4AEpWrs}  with \eqref{oDyfIYIan}
shows that if a random $z$ computes an incompressible proper tree then $z\geq_T\zj$.
Reduction $z\geq_T\zj$ in this indirect proof is non-uniform:
the first case depends on $g(m):=\min\sqbrad{n}{\abs{T\cap 2^n}>2^m}$
and the second is implicit in the difference-test we construct.  
We give a proof where the two reductions for $z\geq_T\zj$ are explicit.

\begin{coro}\label{w9Y9wimsv}
If $z$ is random and computes a proper pruned incompressible tree, $z\geqT\zj$.
\end{coro}\begin{proof}
Let $\Phi$ be a Turing functional, 
$z$ a random real, and  $\Phi(z)$ a proper tree with $\forall i,\ \ds(\Phi(z;i))\leq k_0$ for some $k_0$.
Let $(n_s)$ be a computable enumeration of $\emptyset'$ without repetitions 
and
\begin{align*}
P_i^m(s)\  :=\ &\sqbrad{x}{\forall i\leq s,\ \ds_m(\Phi_s(x;i))\leq k_0}\\[0.1cm]
G^m_i(s)\  :=\  & \sqbrad{x\in P^m_i(s)}{|\Phi_s(x;i)|\geq 2^{n_i}}.
\end{align*}
By Lemma \ref{6xHgSw7WOv}  we have $\mu(\bigcap_m G^m_i(s)) \leqt 2^{-n_i}$  
so there is an increasing computable $f$ such that 
\[
\mu(G\asto_i(s)) \leqt 2^{-n_i}
\hspace{0.3cm}\textrm{where $G\asto_i(s):=G^{f(i,s)}_i(s)$.}
\]
To convert the difference test  
in Theorem \ref{NtI4AEpWrs} into a  \ml test $(V_s)$, 
at stage $s$ we make enumerations with respect to each $n_i\geq n_s, i\leq s$ instead of just $n_s$:
\[
V_s\  :=\   \bigcup\sqbrad{G\asto_i(s)}{i\leq s\wedga n_i\geq n_s}.
\]
The family $(V_s)$ is a Solovay test since the $V_s$ are clopen,   \sz uniformly in $s$, and
\[
\mu(V_s)\leq \sum_i \mu(G_i\asto(s))\leq \sum_{j\geq n_s} 2^{-j}=2^{1-n_s}.
\]
The functions
$h(i):=\min\sqbrad{s}{\Phi_s(z; i)\de}$ and  $g(n):=\min\sqbrad{i}{\abs{\Phi(z;i)}>2^n}$
are computable in $z$. Since $z$ is random, $z\not\in V_s$
for sufficiently large $s$   and 
$z\not\in G\asto_i(s)$ when $i\leq s$, $n_i\geq n_s$.
So
\begin{equation}\label{wvR8mrf3ib}
\parb{g(n_{s})\leq s<s'\wedga n_{s'}<n_{s}\wedga s'>h(s)}\impl z\in V_s.
\end{equation}
Let $P := \sqbrad{x}{\forall i,\ \ds(\Phi(x;i))\leq k_0}$ and
$G_i\  :=\   \sqbrad{x\in P}{|\Phi(x;i)|\geq 2^{n_i}}$.
There are two reasons why $z\not\in G\asto_i(s)$ occurs: (i) when
$g(n_i)>i$:  the size $\abs{\Phi(z;i)}$  is not sufficiently large for $z\in G_i$; and (ii) when
$\Phi_s(z;i)\un$, so then $s$ is smaller than the stage $h(i)$ where $z$ enters $G_i$.
If $g(n_s)\geq s$ for all but finitely many $s$, $z\geq_T g\geqT \zj$. 
Otherwise
there is a $z$-computable increasing $(s_i)$ with $g(n_{s_i})< s_i$.
Since $z\not\in  V_s$ for large $s$, by \eqref{wvR8mrf3ib}: 
$g(n_s)<s<h(s)\leq t\impl n_t>n_s$.
So $h(s_i)$ bounds the settling time of $\zj\cap [0, n_{s_i}]$. Since $h, (s_i)$ are $z$-computable, 
$z\geqT \zj$.
\end{proof}
\begin{rem}
Recall the growth of tree $T$ in the  proof of Theorem \ref{NtI4AEpWrs}  in terms of 
\[
g(m):=\min\sqbrad{n}{\abs{T\cap 2^n}>2^m}
\]
and the dichotomy with respect to whether $g$ bounds the settling time of $\zj$.
The reader may wonder if 
this argument can be reduced to the work of \citet{bslBienvenuP16}, in the sense that
the case where $g$ does not bound the settling time of $\zj$ corresponds to a tree which is a member of a deep \pz class.
In \S\ref{TNQpxFKaL8} we show that this is not so. Indeed, the growth of a tree is not related in any significant way with
the possible membership of it in a deep \pz class.
\end{rem}

Recall that the {\em deficiency} of $\sigma\in\twomel$ and  $E\subseteq\twomel$ is given by 
\[
\ds(\sigma):=|\sigma|-K(\sigma)
\hspace{0.3cm}\textrm{and}\hspace{0.3cm}
\ds(E):=\sup_{\sigma\in E} \ds(\sigma),
\]
and a tree $T$ is \textit{weakly-incompressible} if there is a constant $c$ 
and infinitely many $n$ such that $T\cap 2^n$ has deficiency at most $c$. We prove Theorem \ref{YRH2XAwqKL}.

\begin{thm}\label{JSYt4E3hxL}
Each random $z$ computes a weakly-incompressible perfect tree.
\end{thm}\begin{proof}

We give a probabilistic computable construction of nondecreasing $(\ell_n)$ and $F_n\subseteq 2^{\ell_n}$
and show that  with probability 1, $F:=\bigcup_n F_n$ is incompressible and 
 for sufficiently large $n$ each string in $F_n$ has two extensions in $F_{n+1}$. Let
 \[
 I_n=[0,2^{n^2})
 \hspace{0.3cm}\textrm{and}\hspace{0.3cm}
 \ell_0=0
 \hspace{0.3cm}\textrm{and}\hspace{0.3cm}
 F_0:=\{\lambda\}.
 \]
 For $n>0$ we randomly choose  $\ell_n \in I_n$. If $\ell_n\leq \ell_{n-1}$ we instead set $\ell_n=\ell_{n-1}$ and let $F_n=F_{n-1}$.
Otherwise for each $\tau\in F_{n-1}$ we put into $F_n$ two randomly chosen distinct extensions of $\tau$ of length $\ell_n$.

As in Theorem \ref{9NoMkmNVvP}, the construction is a total Turing functional $\Phi$ 
where each oracle $z$ outputs an array $\Phi(z;n)=F_n(z)$, $F(z):=\cup_n F_n(z)$  and corresponds to a unique choice
$(\ell_n(z)), (F_n(z))$ of $(\ell_n), (F_n)$. It suffices to show that for each random $z$ and sufficiently large $n$,
each string in $F_n(z)$ has two extensions in $F_{n+1}(z)$ and $F(z)$ is incompressible. Then
the downward $\preceq$-closure of $F(z)$ would be a weakly incompressible tree.

Consider the following conditions where $b$ is a constant to be fixed later:
\begin{enumerate}[\hthree(a)]
\item $\ell_n\leq \ell_{n-1}\ \vee\ K(\ell_n-\ell_{n-1})\leq n^2-n/3$
\item $\exists \sigma\in 2^{\ell_n}\cap F(z), \tau\in 2^{\ell_{n-1}}\ \parb{\sigma\succeq\tau\wedga 
K(\sigma\ |\ \tau)\leq \ell_n-\ell_{n-1}+n^2-5n/3+b}$.
\end{enumerate}

Let $V_n$ be the set of reals $z$ with respect to which one of (a), (b) holds, and note that it is a
$\Sigma^0_1$ class uniformly in $n$. We obtain a bound on $\mu(V_n)$.

Since $\ell_n\in [0,2^{n^2})$, $\ell_{n-1}<2^{(n-1)^2}$ the probability of $\ell_n\leq \ell_{n-1}$  is $\leq 2^{(n-1)^2}/2^{n^2}=2^{-2n+1}$.

Assuming $\ell_n> \ell_{n-1}$, since $|\{\ell:K(\ell)\leq t\}|\leq 2^t$ we have 
\[
\abs{\sqbrad{\ell_n}{K(\ell_n-\ell_{n-1})\leq n^2-n/3}} \leq 2^{n^2-n/3} = 2^{-n/3}\cdot \abs{I_n}
\] 
so the probability of (a) is $\leq 2^{-n/3}+2^{-2n+1}$.

Assuming  $\neg$(a), by the counting theorem relativized to $\tau$ for each $\tau\in f(z)\cap 2^{\ell_{n-1}}$:
\begin{align*}
& \abs{\sqbrad{\sigma\in 2^{\ell_n}}{\sigma\succeq\tau \wedga K(\sigma\ |\ \tau)\leq \ell_n-\ell_{n-1}+n^2-5n/3+b}}   \\
=\ &
\abs{\sqbrad{\rho\in 2^{\ell_n-\ell_{n-1}}}{K(\rho\ |\ \tau)\leq \ell_n-\ell_{n-1}+n^2-5n/3+b}}\\
\leqt\ & 2^{\ell_n-\ell_{n-1}-K(\ell_n-\ell_{n-1})+n^2-5n/3+b} \\ 
\leq\ & 2^{\ell_n-\ell_{n-1}}\cdot 2^{-4n/3+b}.
\end{align*}
So the probability of (b) for $\sigma$ is $\leqt 2^{-4n/3+b}$.
By independence, the probability of (b)  is $\leqt 2^n\cdot 2^{-4n/3+b}=2^{-n/3+b}$.  
In total, the probability that  (a) or (b) occurs is 
\[
\mu(V_n)\leqt 2^{-2n+1}+2^{-n/3}+2^{-n/3+b}.
\]
So $\sum_n\mu(V_n)<\infty$ and $(V_n)$ is a Solovay test.

Finally, let $z$ be a random real.
Fix $m$ such that for all $n>m$ none of (a), (b) occurs.
By $\neg$(a) for each $n>m$, each $\sigma\in F_n(z)$ has two extensions in $F_{n+1}(z)$. 
It remains to show that for a sufficiently large $b$ used in the definition of $(F_n)$,
for each  $n>m$ the deficiency of $F(z)\cap 2^{\ell_n}$ is at most $c:=\ds(F(z)\cap 2^{\ell_m})$.

Assuming that this holds for $n-1$, since (b)  does not occur at $n$, each $\sigma\in F_n(z)$ extending $\tau\in F_{n-1}(z)$ satisfies 
$K(\sigma|\tau)>\ell_n-\ell_{n-1}+n^2-5n/3+b$. By the induction hypothesis $K(\tau)\geq |\tau|-c$, combining with $K(\tau)\leqp |\tau|+2\log|\tau|$ we have $K(K(\tau)-|\tau|)\leqp 2\log\log|\tau|$. From $\tau$, $K(\tau)-|\tau|$ we can recover $\tau^*$, so \[
K(\sigma|\tau^*)\geqp K(\sigma|\tau,K(\tau)-|\tau|)\geqp K(\sigma|\tau)-K(K(\tau)-|\tau|)\geqp K(\sigma|\tau)-2\log\log|\tau|.
\]
From $\sigma$, $|\tau|=\ell_{n-1}$ we can recover $\tau$, so $K(\sigma)+K(\ell_{n-1})\geqp K(\sigma,\tau)$ where $K(\ell_{n-1})\leqp\log|\ell_{n-1}|+2\log\log|\ell_{n-1}|\leq(n-1)^2+4\log n$. Then
\begin{align*}
K(\sigma) &\geqp K(\sigma,\tau)-K(\ell_{n-1}) \\
&\eqp K(\sigma|\tau^*)+K(\tau)-K(\ell_{n-1}) \\
&\geqp K(\sigma|\tau)-2\log\log|\tau|+K(\tau)-K(\ell_{n-1}) \\
&\geqp (\ell_n-\ell_{n-1}+n^2-5n/3+b)-4\log n+(|\tau|-c)-((n-1)^2+4\log n) \\
&= |\sigma|-c+b+(n/3-8\log n) \\
&\geqp |\sigma|-c+b.
\end{align*}
So for a sufficiently large $b$ we have $K(\sigma)\geq|\sigma|-c$. This completes the induction step and the proof that  
$F(z)$ is incompressible.
\end{proof}

As expected, the tree constructed in Theorem \ref{JSYt4E3hxL} is less fat than the set of Theorem \ref{9NoMkmNVvP}; 
in particular it is $2^{\sqrt{\log n}}$-fat, as the parameters in the proof indicate.

\section{Nearly incompressible trees}\label{d6GV7C9Wxm}

By Theorem \ref{7i3g6ty9SJ}, the deficiency of a tree $T$ has to increase as its number of paths (or its width) increase.
In this section we explore the relationship between the growth of the width and the deficiency in such trees, and prove Theorems \ref{O65ckpRLDvabb1}, \ref{O65ckpRLDvabb} and \ref{O65ckpRLDvabb3}.

We start with showing a stronger version of Theorem \ref{O65ckpRLDvabb1}.
\begin{prop}\label{LhaG7r8Qyk}
Let $g$ be a computable order with $\forall n,\ g(n+1)\leq g(n)+1$.
Every random real computes a perfect tree $T$ with $\log |T\cap 2^{n}|=g(n)$ and 
\[
\sum_n \frac{2^{\ds(T\cap 2^n)}}{|T\cap 2^n|}<\infty.
\]
In particular, $\lim_n\parb{\log |T\cap 2^n|- \ds(T\cap 2^n)}=\infty$.
\end{prop}\begin{proof}
Given $g,x$ let $f(n):=\min\sqbrad{i}{g(i)\geq n}$ and 
$T$ be the pruned tree whose paths $z$ are the reals that agree with $x$
outside the positions in $f(\Nat)$, namely $\forall n\in \Nat-f(\Nat),\ z(n)=x(n)$. Then 
$g(n)=\abs{\sqbrad{i}{f(i)\leq n}}=\log |T\cap 2^{n}|$.

For each $k\leq f(n)$ and $\sigma\in T\cap 2^{k}$ we can compute $x\restr_k$ given $\sigma$ and $x(f(0)),\cdots,x(f(n-1))$. Therefore 
$K(x\restr_k)\leqp K(\sigma)+n=K(\sigma)+\log \abs{T\cap 2^k}$.
So 
\begin{equation}\label{47Vl8xtza5a}
\ds(\sigma)\leqp \ds(x\restr_k) +\log \abs{T\cap 2^k}
\hspace{0.3cm}\textrm{and}\hspace{0.3cm}
\ds(T\cap 2^k)\leqp \ds(x\restr_k) +\log \abs{T\cap 2^k}.
\end{equation}
By the ample excess lemma \cite{milleryutran2} and the randomness of  $x$ we get $\sum_n 2^{\ds(x\restr_n)}<\infty$ which, in combination with 
\eqref{47Vl8xtza5a}, concludes the proof.
\end{proof}

Therefore the deficiency of $T\cap 2^{\leq n}$ can grow slower than $\log\abs{T\cap 2^{n}}$. To quantify the limits of this gap we need a strengthening of Lemma \ref{DP8sKcDJCf}.

\begin{lem}\label{jOgzCAVZg3}
Given Turing functionals $\Phi, \Theta$ with $\Phi(x;n)\subseteq 2^n$, $\Theta(x;n)\in \Nat$ let 
\begin{align*}
G\ :=&\ \sqbrad{x}{\ds(\Phi(x;n))\leq \Theta(x;n)\text{ for all }n} \\[0.1cm]
Q_{\sigma}^p\ :=&\ \sqbrad{x}{\Theta(x;|\sigma|)= p\wedga \sigma\in \Phi(x;|\sigma|)}.
\end{align*}
Then $\exists c\ \forall n,p\ \forall \sigma\in 2^{n}: \mu(Q_{\sigma}^p\cap G)\leq 2^{c+p-n}$.
\end{lem}\begin{proof}
The argument is based on the following fact:
\[
\parb{x\in Q_{\sigma}^p\wedga \ds(\sigma)>p}\impl x\not\in G.
\]
Note that $G$ is a \pz class and $Q_{\sigma}^p$ is \sz uniformly in $\sigma,p$.
The idea is to 
\begin{itemize}
\item search for $\sigma,p$ such that the approximation to
 $\mu(Q_{\sigma}^p)$ exceeds the purported bound
 \item compress $\sigma$, increasing its deficiency to $p+1$.
\end{itemize}
The latter action ejects the reals in  the current approximation of $Q_{\sigma}^p$ from $G$,
thus reducing $\mu(G)$ by at least the cost $2^{p-|\sigma|}$ of the compression.
 
Let $G(s), Q_{\sigma}^p(s)$ the approximations to $G, Q_{\sigma}^p$ at stage $s$, defined as in the statement but with subscript $s$ in $\Phi, \Psi, \ds$.
We define a \pf machine $M$. By the recursion theorem we may use a constant $c$ in the definition of $M$ such that $K\leq K_M+c$. 

{\bf Construction of $M$.} 
Let $t_{0}:=0$ and for each $i\geq 0$:
\begin{enumerate}[(i)]
\item search for  $n,p\leq n-c-2, \sigma\in 2^n$ and $s_i>t_{i}$ such that 
\[
\mu(Q_{\sigma}^{p}(s_i)\cap G(s_i))> 2^{c+2+p-n}
\]
and  produce an $M$-description of $\sigma$ of length $n-p-c-1$
\item wait for  $t_{i+1}>s_i$ with  
$K_{t_{i+1}}(\sigma)\leq K_{M,s_{i}}(\sigma)+c$ 
\item set $n_i=n, p_i=p, \sigma_i=\sigma$ and go to (i) for $i:=i+1$.
\end{enumerate}

{\bf Verification.}
We first show that the weight of $M$ is  bounded by 1.
At round $i$:
\begin{itemize}
\item at step (i) weight $2^{1+p_i-n_i+c}$ is adde d in $M$, 
increasing the $M$-deficiency of $\sigma_i$ to  $p_i+c+1$
\item if $t_{i+1}$ at step (iii) is reached, the deficiency of $\sigma_i$ increases to  $p_i+1$
\item so $\ds_{t_{i+1}}(\sigma_i)\geq p_i+1$
and $Q_{\sigma_i}^{p_i}(s_i)\cap G(t_{i+1})=\emptyset$.
\end{itemize}
Since $G(s)\subseteq G(s+1)$, each increase $2^{c+1+p_i-n_i}$
in the weight of $M$ is matched by a decrease 
\[
\geq \mu(Q_{\sigma_i}^{p_i}(s_i)\cap G(s_i))> 2^{2+p_i-n_i+c}
\]
of  $\mu(G)$ of size which is twice larger, with at most one exception if the search at step (iii) is indefinite. In the  latter case there will be no more weight added to $M$ and  since $|\sigma_i|>p_i+c+2$, this last contribution to the weight of $M$ is at most 
$2^{-1}$.
So the total weight of $M$ is at most
\[
2^{-1}+\mu(G)/2< 1.
\]
By the Kraft-Chaitin-Levin theorem $M$ is well-defined.
By the choice of $c$ the definition of $M$ will never get stuck at (iii), waiting for the required
stage $t_{i+1}$ indefinitely. 

By step (i),  $\mu(Q_{\sigma}^p\cap G)\leq 2^{c+p-n}$ if $p\leq n-c-2$. Otherwise
$2^{c+p-n}>1\geq \mu(Q_{\sigma}^p\cap G)$.
\end{proof}

We first quantify the limit of the gap between deficiency and width for trees computable by incomplete randoms, 
in terms of computable upper bounds on the deficiency.

\begin{thm}
Let $h$ be a computable order and  $T$ be a tree with 
\[
\ds(T\cap 2^n)\leq h(n)\leq \log \abs{T\cap 2^n}. 
\]
If $z$ is difference random and $T\leqT z$, there is no order $g\leqT z$ with 
 $\log \abs{T\cap 2^n}\geq g(n)+h(n)$.
\end{thm}\begin{proof}
Given $T, h$ be as above, suppose that $T,g\leqT z$ and 
\begin{equation}\label{sFI8yT2fLF}
\log \abs{T\cap 2^n}\geq g(n)+h(n).
\end{equation}
It suffices to show that $z$ is not difference random. Without loss of generality, we assume that $z\not\geqT\zj$.
Let $\Phi, \Psi$ be Turing functionals with $\Phi(z;n)=T\cap 2^n$, $\Psi(z;n)=g(n)$ and set 
\begin{align*}
G\ :=&\ \sqbrad{x}{\forall n,\  \parb{\Phi(x;n)\un\ \vee \ \ds(\Phi(x;n))\leq h(n)}} \\[0.1cm]
Q_{\sigma}\ :=&\ \sqbrad{x\in G}{\sigma\in \Phi(x;|\sigma|)}\\[0.1cm]
Q_{n}^p\ :=&\ \sqbrad{x\in G}{\Psi(x;n)\geq p\wedga \log \abs{T\cap 2^n}\geq \Psi(x;n)+h(n)}.
\end{align*}
By Lemma \ref{jOgzCAVZg3} for $\Theta(n):=h(n)$ there exists $c$ with 
$\forall \sigma\in 2^{n}: \mu(Q_{\sigma})\leq 2^{c+h(n)-n}$, so
\[
\sum_{\sigma\in 2^n} \mu(Q_{\sigma})\leq 2^{c+h(n)}.
\]
By \eqref{sFI8yT2fLF} each $x\in Q_{n}^p$ belongs to at least $2^{p+h(n)}$ many
 of the $Q_{\sigma}, \sigma\in 2^n$ so
\[
\mu(Q_{n}^p)\leq 2^{-p-h(n)}\cdot \sum_{\sigma\in 2^n} \mu(Q_{\sigma})\leq 2^{c-p}.
\]
Let $f(p):=\min\sqbrad{n}{g(n)>p}$ so $f\leqT z$. By the hypothesis, $z\in G$ and  
\begin{equation}\label{oo4rXKdPsa}
n> f(p)\impl z\in Q_{n}^p.
\end{equation}
Let $(n_s)$ be a computable enumeration of $\emptyset'$ without repetitions, so 
$(Q_{s}^{n_s})$ is a Solovay difference test.
Since $z\not\geqT\zj$, there exist infinitely many $s$
with $s> f(n_s)$.  By \eqref{oo4rXKdPsa} we have  $z\in Q_{s}^{n_s}$ for infinitely many $s$, 
so $z$ is not difference random.
\end{proof}

A similar argument gives Theorem \ref{O65ckpRLDvabb3} from \S\ref{yP4cTWTJZG}, regarding trees with computable width.

\begin{thm}
Suppose that $T$ is a tree and $(\abs{T\cap 2^n})$ is computable. If $z$ is difference random 
and $T\leqT z$, there is no order $g\leqT z$ with 
 $\log \abs{T\cap 2^n}\geq g(n)+\ds(T\cap 2^n)$.
\end{thm}\begin{proof}
Given $T$ be as above, suppose that $T,g\leqT z$ and 
\begin{equation}\label{sFI8yT2fLFab}
\log \abs{T\cap 2^n}\geq g(n)+\ds(T\cap 2^n).
\end{equation}
It suffices to show that $z$ is not difference random. Without loss of generality, we assume that $z\not\geqT\zj$.
Let $\Phi, \Psi$ be Turing functionals with $\Phi(z;n)=T\cap 2^n$, $\Psi(z;n)=g(n)$ and set 
\begin{align*}
G\ :=&\ \sqbrad{x}{\forall n,\ \ds(\Phi(x;n))\leq \log \abs{T\cap 2^n}- \Psi(x;n)} \\[0.1cm]
Q_{n}^p\ :=&\ \sqbrad{x\in G}{\Psi(x;n)\geq p\wedga \abs{T\cap 2^n}=\Phi(x;n)}\\[0.1cm]
Q^p_{\sigma}\ :=&\ \sqbrad{x\in Q_{|\sigma|}^p}{\sigma\in \Phi(x;|\sigma|)}.
\end{align*}
By Lemma \ref{jOgzCAVZg3} for $\Theta(x; n):=\log\abs{T\cap 2^n}- g(n)$ there exists $c$ with
\[
\forall \sigma\in 2^{n}: \mu(Q^p_{\sigma})\leq 2^{c+\log\abs{T\cap 2^n}-p-n}
\hspace{0.3cm}\textrm{so}\hspace{0.3cm}
\sum_{\sigma\in 2^n} \mu(Q^p_{\sigma})\leq 2^{c+\log\abs{T\cap 2^n}-p}.
\]
By \eqref{sFI8yT2fLFab} each $x\in Q_{n}^p$ belongs to $\abs{T\cap 2^n}$ many 
of the $Q^p_{\sigma}, \sigma\in 2^n$ so
\[
\mu(Q_{n}^p)\leq 2^{-\log \abs{T\cap 2^n}}\cdot \sum_{\sigma\in 2^n} \mu(Q^p_{\sigma})\leq 2^{c-p}.
\]
Let $f(p):=\min\sqbrad{n}{g(n)>p}$ so $f\leqT z$. By the hypothesis $z\in G$ and  
\begin{equation}\label{oo4rXKdPsab}
n> f(p)\impl z\in Q_{n}^p.
\end{equation}
Let $(n_s)$ be a computable enumeration of $\emptyset'$ without repetitions, so 
$(Q_{s}^{n_s})$ is a Solovay difference test.
Since $z\not\geqT\zj$, there exist infinitely many $s$
with $s> f(n_s)$.  By \eqref{oo4rXKdPsab} we have  $z\in Q_{s}^{n_s}$ for infinitely many $s$, 
so $z$ is not difference random.
\end{proof}

Finally we give a general bound, which has
Theorem \ref{O65ckpRLDvabb} from \S\ref{yP4cTWTJZG} as a special case.

\begin{thm}\label{O65ckpRLDv}
Suppose that $g$ is a computable order and $\sum_n 2^{-g(n)}$ converges to a computable real.
Then no difference random real can compute a tree $T$ with 
\begin{equation}\label{sFI8yT2fLFabd}
\log \abs{T\cap 2^n}\geq g(\log \abs{T\cap 2^n})+ \ds(T\cap 2^n).
\end{equation}
\end{thm}\begin{proof}
Given $T,g$ be as above and $z\geqT T$ we show that $z$ is not difference random. 

Let $\Phi$ be a Turing functional with $\Phi(z;n)=T\cap 2^n$ and  
\begin{align*}
G\ :=&\ \sqbrad{x}{\forall n,\ \ds(\Phi(x;n))\leq \log\abs{\Phi(x;n)}- g(\log\abs{\Phi(x;n)})} \\[0.1cm]
Q_{n}^p\ :=&\ \sqbrad{x\in G}{\log\abs{\Phi(x;n)}=p}\\[0.1cm]
Q^p_{\sigma}\ :=&\ \sqbrad{x\in Q_{|\sigma|}^p}{\sigma\in \Phi(x;|\sigma|)}.
\end{align*}
By Lemma \ref{jOgzCAVZg3} for $\Theta(n):=\log\abs{\Phi(x;n)}- g(\log\abs{\Phi(x;n)})$ there exists $c$ with
\[
\forall \sigma\in 2^{n}: \mu(Q^p_{\sigma})\leq 2^{c+p-g(p)-n}
\hspace{0.3cm}\textrm{so}\hspace{0.3cm}
\sum_{\sigma\in 2^n} \mu(Q^p_{\sigma})\leq 2^{c+p-g(p)}.
\]
By \eqref{sFI8yT2fLFabd} each $x\in Q_{n}^p$ belongs to $2^p$ many 
of the $Q^p_{\sigma}, \sigma\in 2^n$ so
\[
\mu(Q_{n}^p)\leq 2^{-p}\cdot \sum_{\sigma\in 2^n} \mu(Q^p_{\sigma})\leq 2^{c-g(p)}.
\]
Since $g, \sum_i 2^{-g(i)}$ are computable, so  is 
$h(n):=\min\sqbrad{n}{\sum_{i\geq k} 2^{-g(i)}<2^{-n-c}}$.
Let 
\[
E_n^{k}:=\bigcup_{p\geq h(k)} Q_n^p
\hspace{0.3cm}\textrm{so}\hspace{0.3cm}
\mu(E_n^{k})\leq \sum_{p\geq h(k)} \mu(Q_{n}^p) < 2^{-k}.
\]
Let $f(k):=\min\sqbrad{n}{\log\abs{\Phi(x;n)}>h(k)}$ so $f\leqT z$.
By the hypothesis $z\in G$ and  
\begin{equation}\label{oo4rXKdPsabd}
n> f(k)\impl z\in E_{n}^k.
\end{equation}
Let $(n_s)$ be a computable enumeration of $\emptyset'$ without repetitions, so 
$(E_{s}^{n_s})$ is a Solovay difference test.
Since $z\not\geqT\zj$, there exist infinitely many $s$
with $s> f(n_s)$.  By \eqref{oo4rXKdPsabd} we have  $z\in E_{s}^{n_s}$ for infinitely many $s$, 
so $z$ is not difference random.
\end{proof}

Theorem \ref{O65ckpRLDv} also holds when  
$\sum_n 2^{-g(n)}$ is merely non-random, rather than computable. 
We omit this proof as it  only requires a  standard adaptation of the above argument.

\begin{rem}
Consider a \pf machine $M$ with $K_M(n)=2\log n$ and let us refer to the deficiencies $\abs{\sigma}-K_M(\sigma)$ 
as {\em typical} (as they correspond to a roughly optimal computable information content measure).
Then  $\log \abs{T\cap 2^n}$ is typically compressed by $\log \abs{T\cap 2^n}-2\log\log \abs{T\cap 2^n}$ bits.
So for $g(n):=2\log n$,  Theorem \ref{O65ckpRLDv}  says that some   $\sigma\in T\cap 2^n$ are compressed more than the typical compression of the 
logarithm $\log \abs{T\cap 2^n}$ of the corresponding width $\abs{T\cap 2^n}$ of $T$.
We do not know if this is true with respect to the universal compression, namely if incomplete randoms fail to compute trees
$T$ with $\ds(T\cap 2^n)\leq \ds(\log\abs{T\cap 2^n})$. The latter property is 
the modification of \eqref{sFI8yT2fLFabd} with $K$ in place of $g$.
\end{rem}
\section{Negligibility and depth of classes of reals}\label{vbNyewP8tQ}

In this section we review the notion of negligibility and depth. A weaker version of our Theorem \ref{7i3g6ty9SJ} can be derived from the work on deep $\Pi^0_1$ classes, so we investigate if this is also the case for Theorem \ref{7i3g6ty9SJ} itself. In particular we prove Theorems \ref{FJCm7I5aYa} and \ref{41rMOJConYa} which show that Theorem \ref{7i3g6ty9SJ} escapes this scope. 

Negligible classes were studied systematically in \cite{neglivyuginold} 
and earlier in \cite{leviniandcLevin84, neglivyuginLevin, MR0316227}.

\begin{defi}\label{TsHLkveM5D}
We say that $\CC\subseteq\twome$ is 
\begin{itemize}
\item {\em negligible} if $\mu(\sqbrad{y}{\exists x\in \CC,\ x\leqT y})=0$
\item {\em $tt$-negligible} if $\mu(\sqbrad{y}{\exists x\in \CC,\ x\leq_{tt} y})=0$.
\end{itemize}
\end{defi}

So  $\CC$ is negligible iff $\mu(\Phim(\CC))=0$ for all Turing functionals $\Phi$.
Similarly,
\begin{equation}\label{2yoZgaEmEK}
\textrm{$\CC$ is $tt$-negligible iff $\mu(\Phim(\CC))=0$ for all total Turing functionals $\Phi$}
\end{equation}
where $\Phi$ is total iff $\Phi(x;n)$ is defined for each oracle $x$ and each $n$.
\citet{bslBienvenuP16} defined
negligibility equivalently in terms of the 
universal continuous \lce semimeasure $\Mb$ from \cite{MR0307889}. Our
Definition \ref{TsHLkveM5D} is their
 \cite[Proposition 3.2]{bslBienvenuP16}.
By \cite{MR0307889} every Turing functional $\Phi$ corresponds to the \lce semimeasure 
$\mu_{\Phi}(\sigma):=\mu(\Phim(\sigma))$
and vice-versa, for every \lce semimeasure $\nu$ there exists a Turing functional $\Phi$ with
$\nu=\mu_{\Phi}$. In the case of Turing functionals which are total on a set of oracles of measure 1,
$\mu_{\Phi}$ is a computable measure and a similar equivalence holds.
So \eqref{2yoZgaEmEK} implies that
\[
\textrm{$\CC$ is {\em $tt$-negligible} if $\nu(\CC)=0$ for every computable measure $\nu$.}
\]
\citet{bslBienvenuP16} considered a strengthening of negligibility for \pz classes based on the notion of
{\em logical depth} by \citet{Bennett1988}. A real is {\em deep} if the compression of its initial 
segments with respect to any computable time-bound on the running time of the
universal machine is {\em far} from their optimal compression (Kolmogorov complexity).
Bennett's  {\em  slow growth law} (shown in \cite[Theorem 1]{Bennett1988} and more generally in \cite{StephanMoserdeep}) 
says that the  initial segments of deep sequences cannot be produced via probabilistic computation.

According to \cite[Definition 4.1]{bslBienvenuP16}, a \pz class $P$ is {\em deep} if
there is a computably vanishing upper bound on the probability that  the universal
Turing functional produces the set $P_n$ of  $n$-bit prefixes of the reals in this class. They showed that this is equivalent to the existence of a computable order $g$ such that $\mm(P_n)\leq 2^{-g(n)}$, where $\mm$ is the a priori discrete probability.
The same is true with the a priori continuous semimeasure $\Mb$ in place of $\mm$.
Several examples of deep \pz classes were given in \cite{bslBienvenuP16}, including
the class of codes of complete extensions of $\PA$ and the class of 
$g$-fat incompressible trees, where $g$ is a computable order.

\subsection{Negligibility  beyond effectively closed classes}
We show that two facts from \cite{bslBienvenuP16} concerning negligible \pz classes hold
more generally for negligible \pzt classes.
By \cite[Theorem 5.1, 5.2]{bslBienvenuP16}:
\begin{enumerate}[\hthree(a)]
\item weakly 1-randoms   do not $tt$-compute members of $tt$-negligible \pz classes. 
\item weakly 2-randoms  do not compute members of negligible \pz classes.
\end{enumerate}
We show the analogue of (a) for weakly 2-randoms  and extend (b) to \pzt classes.

\begin{thm}
Weakly 2-random reals:
\begin{enumerate}[(i)]
\item  do not $tt$-compute members of $tt$-negligible \pzt classes.
\item  do not compute members of negligible \pzt classes.
\end{enumerate}
\end{thm}\begin{proof}
Let $\CC=\bigcap_i Q_i$ be a \pzt class, where the $Q_i$ are uniformly \sz sets of reals, 
and $\Phi$ be a Turing functional. 
Since $\Phi$ is effectively continuous, the $\Phim(Q_i)$ are uniformly \szn. Then  
\[
\Phim(\CC)=\Phim\left(\bigcap_i Q_i\right)=\bigcap_i \Phim(Q_i).
\]
This shows that $\Phim(\CC)$ is a \pzt class. 
If $\CC$ is $tt$-negligible then $\mu(\Phim(\CC))=0$ for each total $\Phi$, so any
real that $tt$-computes a member of $\CC$ fails to be weakly 2-random.
If $\CC$ is negligible,  $\mu(\Phim(\CC))=0$ so any real that computes 
a member of $\CC$ fails to be weakly 2-random.
\end{proof}
\subsection{Depth beyond effectively closed classes}\label{TNQpxFKaL8} 
A key property of deep \pz classes was established in \cite[Theorem 5.3]{bslBienvenuP16}: 
\begin{equation}\label{1UpFa6eE3Z}
\textrm{a random computes a member of a deep \pz class iff it computes $\zj$.}
\end{equation}
Then as examples of deep $\Pi^0_1$ classes, they considered the class $\mathcal{K}_{f,\ell,d}$ where $\ell,f,d$ are computable functions and $\mathcal{K}_{f,\ell,d}$ consists of sequences $\vec{F}=(F_1,F_2,F_3,\dots)$ where for all $i$, $F_i$ is a finite set of $f(i)$ strings $\sigma$ of length $\ell(i)$ such that $K(\sigma)\geq \ell(i)-d(i)$. They then proved that if $\ell$ is increasing and $f(i)/2^{d(i)}$ takes arbitrarily large values, then $\mathcal{K}_{f,\ell,d}$ is a deep $\Pi^0_1$ class \citep[Theorem 7.7]{bslBienvenuP16}. We can show the following results using the notion of deep \pz classes. 

Recall that a tree $T$ is 
{\em proper} if $\abs{T\cap 2^{k}}$ is unbounded, and 
{\em effectively-proper} if $\abs{T\cap 2^{k}}\geq g(k)$ for a computable order $g$. A real $z$ is {\em computably dominated} if every function $f\leqT z$ is dominated by a computable function.

\begin{coro}[after \cite{bslBienvenuP16}]\label{sjbcegRm4f}
Suppose that $z$ is random.
\begin{enumerate}[(i)]
\item If $z\not\geq_T\zj$ then $z$ does not compute any effectively-proper pruned incompressible tree.  
\item Computably dominated randoms do not compute proper pruned incompressible trees.
\end{enumerate}
\end{coro}\begin{proof}
For (i), let $z$ be random real and suppose that $T$ is a proper pruned incompressible tree with $T\leq_T z$. Consider the sequence $(T\cap 2^i)$. Let $c$ be the deficiency of $T$, and set $f(i)=|T\cap 2^i|$, $\ell(i)=i$ and $d(i)=c$. Then $(T\cap 2^i)$ is a member of $\mathcal{K}_{f,\ell,d}$, and by \citep[Theorem 7.7]{bslBienvenuP16} $\mathcal{K}_{f,\ell,d}$ is a deep $\Pi^0_1$ class. Since $z$ computes $T$, by \cite[Theorem 5.3]{bslBienvenuP16} $z$ cannot be incomplete random. 

For (ii) suppose that $z$ is computably dominated and $Q\leqT z$ is a proper pruned incompressible tree.
Then $z\not\geqT\zj$ and $h(k):=\abs{Q\cap 2^{k}}$ is computable in $z$. Since $z$ is computably dominated, $h$ is lower-bounded by a computable order. So
$Q$ is effectively-proper, contradicting (i).
\end{proof}

This was generalized to all proper incompressible trees in 
Theorem \ref{7i3g6ty9SJ}. It is natural to ask if the latter can be derived from  \eqref{1UpFa6eE3Z}.
This is not the case.  We first show that no proper tree belongs to a \pz class of proper trees, unless it is effectively-proper.
Let $\TT$ denote the set of  pruned trees. 
Since the closed sets of $\twome$ are representable as  pruned trees, 
$\TT$ can be viewed as the set of closed sets in the Cantor space.
A finite tree $F$ is a  {\em prefix of $T$}, denoted by $F\prec T$,  
if $F=T\cap 2^{\leq n}$ for some $n$.
 The {\em prefix topology}  on $\TT$ is generated by the basic open sets 
\begin{equation}\label{oHkQ4KP8LS}
\dbra{F}:=\sqbrad{T\in\TT}{F\prec T}
\end{equation}
indexed by the tree-prefixes $F$. This
coincides with the {\em hit-or-miss} and {\em Vietoris} topologies on the closed sets of the Cantor space 
and $\TT$ is compact \cite[Appendix B]{Molchanov}.
Fix an effective coding of  tree-prefixes into $\twomel$ which preserves the prefix relations,
where the length of the codes of tree-prefixes $T\cap 2^{\leq n}$  is given by a computable 
$n\mapsto p_n$. This induces an effective homeomorphism of $\TT$ onto an effectively closed subspace of $\twome$, and allows to consider \pz classes in $\TT$.
In this way, the \pz subsets of  $\TT$ can be viewed  \pz classes of $\twome$ and vice-versa.
\begin{prop}\label{kpmkMWfCp8}
In the space $\TT$ of pruned trees with the prefix topology, for each $c$:
\begin{enumerate}[(i)]
\item the class  of $c$-incompressible pruned proper trees is not closed
\item if $Q$ is a \pz class of proper pruned trees, there exists a computable order $g$ such that all trees in $Q$ are $g$-proper.
\end{enumerate}
\end{prop}\begin{proof}
For (i), let $J_c$ be the class  of $c$-incompressible pruned proper trees. Then
\begin{itemize}
\item $\TT-J_c$ contains a $c$-incompressible tree $T$ consisting of a single path
\item any neighborhood around $T$ intersects $J_c$. 
\end{itemize}
So $\TT-J_c$ is not open and $J_c$ is not closed.

For (ii) assume that $Q$ is a \pz class of proper pruned trees and let $Q_n$
be the class of trees that are prefixed by $T\cap 2^{\leq n}$ for some $T\in Q$.
Then $Q_{n+1}\subseteq Q_n$ and by the compactness of $\TT$ we have $Q=\bigcap_n Q_n$.
We claim that there exists increasing $(n_k)$ such that 
\begin{equation}\label{mCiXyCuR2w}
\forall k\ \forall T\in Q_{n_k}:\ \abs{T\cap 2^{n_k}}>k.
\end{equation}
Assuming otherwise, there exists $k$ such that:
\begin{itemize}
\item the class $P$ of $T\in\TT$ with at most $k$ many paths is \pz
\item $P\cap Q_n$ are  \pz classes and  $P\cap Q_{n+1}\subseteq P\cap Q_n$   
\item  $\forall n,\ P\cap Q_n\neq\emptyset$ so by compactness  $P\cap Q\neq\emptyset$
\end{itemize}
which contradicts the hypothesis on $Q$.
So the $n_k$ exist and can be searched effectively. 
Then \eqref{mCiXyCuR2w} holds for a computable increasing $(n_k)$ 
and each  $T\in Q$ is $g$-proper for $g(k)=n_k$.
\end{proof}

Proposition \ref{kpmkMWfCp8} does not eliminate the possibility that every
proper incompressible tree might be a member of a deep \pz class of trees (containing non-proper trees).
The latter is excluded by the following, which includes Theorem \ref{FJCm7I5aYa}, where $\Pi^0_1(\zj)$ means \pz relative to $\zj$.

\begin{thm}\label{FJCm7I5aY}
There exists a perfect incompressible tree $T\leqT\zjj$ which is not a member of any deep \pz class.
Moreover  $T$ can be chosen so that $\{T\}$ is a $\Pi^0_1(\zj)$ class.
\end{thm}
The required tree $T$ will be obtained by `shattering' a random real $x$ on an infinite set $A$ of positions so that $[T]=\sqbrad{y}{\forall i\not\in A,\ y(i)=x(i)}$. By the work of \cite{Indiffexn101} on {\em indifferent sets} we can chose $A$
so that this  $T$ is incompressible.
In the following  we  identify \pz classes in  $\twome$ and the space $\TT$ of pruned trees, 
via a fixed effective homeomorphism between them.

\begin{lem}\label{qi87Jc6X9O}
If $x\not\geqT\zj$ is random there exists  $A$ such that  the  pruned tree $T$ with
\[
[T]=\sqbrad{y}{\forall i\not\in A,\ y(i)=x(i)}
\]
 is perfect, incompressible and is not a member of any deep \pz class.
Moreover $A$ can be chosen to be a $\Pi^0_1(x')$ set.
\end{lem}\begin{proof}
Let $(R_e)$ be an effective list of all \pz classes and recall that  $R_e$ is deep iff
\[
\textrm{there exists a computable order $f$ with $\forall n,\ \mm(R_e\cap 2^n)< 2^{-f(n)}$}
\]
which is a $\Sigma^0_3$ statement. So there exists a $\zjj$-computable enumeration of the indices $e$ with the above property.
Fix $h\leqT\zjj$ such that $(R_{h(e)})$ is a list of all deep \pz classes.

Recall from \eqref{oHkQ4KP8LS} that $\dbra{E}$ is the basic open set in $\TT$ corresponding to tree-prefix $E$.

Since $R_{h(e)}$ is a closed, if $Q\not\in R_{h(e)}$ some neighborhood around $Q$ is disjoint from  $R_{h(e)}$. 
Let $x\not\geqT\zj$ be random. Then
by \eqref{1UpFa6eE3Z} for each $e$:
\begin{equation}\label{MADj6DIWrRa}
\textrm{each pruned tree $Q\leqT x$ has a prefix $E$ with $\dbra{E}\cap R_{h(e)}=\emptyset$.}
\end{equation}
By \cite[Theorem 9]{Indiffexn101} there exists infinite $B\leqT x'$ such that pruned tree $G$ with  
\[
[G]:=\sqbrad{y}{\forall i\not\in B,\ y(i)=x(i)}
\]
is incompressible. We define the required $A$ as a subset of $B$. 
If $E\subseteq B$ is finite:
\[
\textrm{the $Q\in\TT$ with $[Q]=\sqbrad{y}{\forall i\not\in E,\ y(i)=x(i)}$ has $Q\equivT x\not\geqT\zj$.}
\]
By \eqref{MADj6DIWrRa} we can  define increasing $(n_e)\leqT B\oplus h\oplus x \leqT x'\oplus \zjj$ and let
\begin{itemize}
\item $A:=\sqbrad{b_{n_e}}{e\in\Nat}$ where $(b_i)$ is the increasing enumeration of $B$
\item $T$ be the pruned tree with $[T]=\sqbrad{y}{\forall i\not\in A,\ y(i)=x(i)}$
\end{itemize}
so that  $\forall e,\ \dbra{T\cap 2^{\leq b_{n_e}}}\cap R_{h(e)}=\emptyset$.
Then $T$ is perfect, $[T]\subseteq [G]$ and
\[
\forall e,\ T\not\in R_{h(e)}
\hspace{0.3cm}\textrm{and}\hspace{0.3cm}
\textrm{$T$ is incompressible.}
\]
Finally we modify the above argument in order to make $A\in\Pi^0_1(x')$.
Let $h, B, (b_i)$ be as before.
Since $h\leqT\zjj$ there exists a $\zj$-computable
$(s,e)\mapsto h_s(e)$  with $h(e)=\lim_s h_s(e)$.

By a finite injury argument based on the changes of the approximations
$h_s(e)\to h(e)$ we can define monotonically movable markers $n_e(s)$ such that the limits $n_e:=\lim_s n_e(s)$ exist and
\[
\textrm{$A:=\sqbrad{b_{n_e}}{e\in\Nat}$ is $\Pi^0_1(x')$}
\hspace{0.3cm}\textrm{and}\hspace{0.3cm}
\forall e,\ \dbra{T\cap 2^{\leq b_{n_e}}}\cap R_{h(e)}=\emptyset
\]
where $T$ be the pruned tree with $[T]=\sqbrad{y}{\forall i\not\in A,\ y(i)=x(i)}$.
\end{proof}
If $A\subseteq\Nat$ is  $\Pi^0_1(x')$, the pruned tree $T$ with 
$[T]=\sqbrad{y}{\forall i\not\in A,\ y(i)=x(i)}$ is $\Pi^0_1(x')$.
So  Theorem \ref{FJCm7I5aY} follows from Lemma \ref{qi87Jc6X9O} by taking $x$ to be a random with $x'\leqT \zj$.

We show that the class of incompressible perfect trees which are not members of 
any deep \pz class is topologically large, namely Theorem \ref{41rMOJConYa} from \S\ref{yP4cTWTJZG}.
This requires a different initial segment construction, based on the following technical fact.
Given $c$ consider the \pz class  
\[
P_c:=\sqbrad{x}{\forall \rho\prec x,\ K(\rho)\geq |\rho|-c}
\]
and let $\ast$ denote concatenation between two strings or a string and a real.
\begin{lem}\label{jaHz524fAy}
Given $c,\ell$ and a set $Q\subseteq 2^{\ell}$ of strings which have extensions in $P_c$, there exist 
\begin{itemize}
\item finite extensions $\tau_{\sigma}, \sigma\in Q$ of the strings in $Q$
\item a random real $z$ with $z'\equivT \zj$
\item a pruned tree $T_z$ with $[T_z]=\sqbrad{\tau_{\sigma}\ast z}{\sigma\in Q} \subseteq P_c$ and $T_z\equivT z$.
\end{itemize}
\end{lem}\begin{proof}
We show that 
there exist   $\tau_{\sigma}, \sigma\in Q$ of the same length  such that $\tau_{\sigma}\succ\sigma$ and
\[
\mu(\sqbrad{x}{\forall\sigma\in Q,\ \tau_{\sigma}\ast x\in P_c})>0.
\]
Since $\sqbrad{x}{\forall\sigma\in Q,\ \tau_{\sigma} \ast x\in P_c}$ is a \pz class, by the low basis theorem this implies the  lemma.
By Lebesgue density  there exist extensions $\tau_{\sigma}$ of the $\sigma\in Q$ of the same length 
with  
\[
\mu_{\tau_{\sigma}}(P_c)>1-1/(2|Q|)
\hspace{0.3cm}\textrm{so}\hspace{0.3cm}
\forall \sigma\in Q,\ \mu(\sqbrad{x}{\tau_{\sigma}\ast x\in P_c})>1-1/(2|Q|).
\]
It follows that $\mu(\sqbrad{x}{\forall\sigma\in Q,\ \tau_{\sigma} \ast x\in P_c})\geq 1-|Q|\cdot 1/(2|Q|)>0$.
\end{proof}
Recall from \cite[\S 14]{Soare16book} or \cite{YatesBM1976} that comeager sets in a topological space are characterized terms of the 
Banach-Mazur game. Inside  the space $\Ts_c$ of pruned subtrees of $P_c$ with the prefix topology, 
two players construct $T\in\Ts_c$ by alternate choices of initial segments (extensions of the current prefix) 
starting from $\{\lambda\}$. A set $\Ws\subseteq \Ts_c$ determines the winner: player 1 wins when $T\in\Ws$.
Then $\Ws$ is comeager in $\Ts_c$ iff player 1 has a winning strategy.
We consider the game where $\Ws$ is the set of perfect trees in $\Ts_c$ that do not belong to any 
 deep \pz class.
\begin{thm}\label{41rMOJConY}
For each $c$, the class of $c$-incompressible perfect trees which are not members of 
any deep \pz class is comeager in the space of $c$-incompressible trees.
\end{thm}\begin{proof}
Consider the Banach-Mazur game inside $\Ts_c$ discussed above. 
Let $T_s$ be the prefix of $T$ at the end of stage $s$.
So $T_0=\{\lambda\}$ and the  $T_{2s+1}\succ T_{2s}$ are chosen by player 1.

Let $(R_{h(e)})$ be a list of all deep \pz classes, as in the proof of Lemma \ref{qi87Jc6X9O}.

By Lemma \ref{jaHz524fAy} and \eqref{MADj6DIWrRa}, any  prefix  $T_{2s}$ of a tree in $\Ts_c$
can be extended to a prefix $T_{2s+1}$ of a tree in $\Ts_c$ such that no suffix of $T_{2s+1}$
in $\Ts_c$ is a member of $R_{h(s)}$. This gives a winning strategy to player 1, who can ensure 
$\forall e,\ T\not\in R_{h(e)}$ irrespectively of the choices $T_{2s}$ of player 2.
It follows that the class of $T\in T_c$ with $\forall e,\ T\not\in R_{h(e)}$  is comeager in $T_c$.
 \end{proof}
\begin{rem}
The initial segment construction of $T$ in  Theorem \ref{41rMOJConY} can be routinely combined
with lowness requirements, showing that there exists a perfect incompressible pruned tree $T$ with $T'\equivT \zj$, which
is not a member of any deep \pz class. 
So the incompressible trees which are not members of deep \pz classes may not compute $\zj$; in particular,
their growth $g(n):=\min\sqbrad{s}{\abs{T\cap 2^s}>n}$ may not compute $\zj$.
 \end{rem}

\section{Positive incompressible trees}\label{VW3LYXitIA}

Recall that a tree $T$ is {\em positive} if $\mu([T])>0$; equivalently $2^n=\bigo{\abs{T\cap 2^n}}$. 

We prove Theorem \ref{t6A4XwWbKX}:
there exists a positive incompressible tree (with deadends) which does not compute any random real.
This fact is implicit in \citet{luliumajma} and makes use of 
\[
\wgt{X}=\sum_{(n,m)\in X}2^{-n} 
\hspace{0.3cm}\textrm{where $X\subset\Nat\times\Nat$.}
\] 
	
\begin{thm}[Theorem 2.1 in \cite{luliumajma}]\label{9pCjcd26em}
There exists $G\subset\Nat\times\Nat$, $\wgt{G}<\infty$ such that:
\begin{itemize}
\item  for every \ce set $D\subset\Nat\times\Nat$, $\wgt{D}<\infty$ we have $\abs{D-G}<\infty$
\item  and  $G$ does not compute any random real.
\end{itemize}
\end{thm}

We point out that the proof of
\cite[Proposition 2.3]{luliumajma} establishes  a stronger result:

\begin{prop}\label{qbqxquZBUY}
There exists a \ce set $D\subset\Nat\times\Nat$,  $\wgt{D}<\infty$ such that if $D\subset X$ and $\wgt{X}<\infty$ then $X$ computes a positive incompressible tree.
\end{prop}\begin{proof}
Let $\langle\cdot,\cdot\rangle$ be a computable bijection from $2^{<\omega}\times\Nat$ to $\Nat$. Let 
\[
D=\bigcup_{k}\{(|\sigma|,\langle\sigma,k \rangle):K(\sigma)\leq|\sigma|-k\}
\] 
and note that $D$ is a \ce set with  
\[
\wgt{D}=\sum_k\sum_{K(\sigma)\leq|\sigma|-k}2^{-|\sigma|}\leq
\sum_k\sum_\sigma 2^{-K(\sigma)-k}\leq\sum_k 2^{-k}<\infty.
\] 
Suppose that $X\supset D$, $\wgt{X}<\infty$ and let 
$T_k=\{\rho:(|\sigma|,\langle\sigma,k\rangle)\notin X\text{ for all }\sigma\prec\rho\}$.
 
Then $T_k$ is an $X$-computable tree. Also $T_k$ is incompressible: if $\sigma$ has deficiency at least $k$ then 
$(|\sigma|,\langle\sigma,k\rangle)\in D\subset X$ so $\sigma\notin T_k$. 
To show that some $T_k$ is positive, consider 
\[
2^\omega-[T_k]=\{x:(|\sigma|,\langle\sigma,k\rangle)\in X
\text{ for some }\sigma\prec x\}=
\bigcup_{(|\sigma|,\langle\sigma,k\rangle)\in X}\llbracket\sigma\rrbracket.
\]
Then
\[
\sum_k\mu(2^\omega-[T_k])\leq
\sum_k\sum_{(|\sigma|,\langle\sigma,k\rangle)\in X} 
\mu(\llbracket\sigma\rrbracket) \leq \wgt{X} < \infty.
\] 
Therefore there must be some $k$ such that $\mu(2^\omega-[T_k])<1/2$, as desired.
\end{proof}

We can now prove Theorem \ref{t6A4XwWbKX}: the
existence of a positive incompressible tree that does not compute any random real.
Let $D$ be as in  Proposition \ref{qbqxquZBUY}  and $G$ as in  Theorem \ref{9pCjcd26em}, so: 
\begin{itemize}
\item $D\cup G\equiv_T G$, since $D-G$ is finite
\item $D\subset D\cup G$ and $\wgt{D\cup G}\leq\wgt{D}+\wgt{G}<\infty$
\item $D\cup G$ computes a positive incompressible tree $T$.
\end{itemize}
Since $T\leq_T G$, it follows that $T$ does not compute any random real.

\section{Conclusion}
We have studied the loss of randomness that occurs 
in multi-dimensional structures such as trees and arrays that are effectively generated from a random oracle.
We focused on examples that go beyond the framework of deep \pz classes of 
\citet{bslBienvenuP16} and demonstrated the limitation of restricting to members of deep \pz classes.
Without generalizing the notion of depth to larger levels of the arithmetical hierarchy we showed that a larger class of trees exhibits
key properties established for members of deep \pz classes in \cite{bslBienvenuP16}.
Similarly, we showed that some properties established in \cite{bslBienvenuP16} for {\em negligible} $\Pi^0_1$ classes
hold more generally for {\em negligible} $\Pi^0_2$ classes. In this fashion we ask
\[
\textrm{{\em Question.} How can depth be defined for $\Pi^0_2$ or higher arithmetical classes?}
\]
It is possible to formalize depth for $\Pi^0_2$ classes in a way that preserves the properties  established
in  \cite{bslBienvenuP16}. However obtaining a natural and  general definition seems non-trivial. 

A specific aspect in this work has been the relationship between the growth of the number of paths of a tree $T$ and the
necessary increase of their randomness deficiency that it necessitates when $T$ is computable by an incomplete random.
As discussed in \S\ref{d6GV7C9Wxm}, in such cases the deficiency $\ds(\log\abs{T\cap 2^n})$ of the width of 
$T$ at level $n$ appears to be the 
lower bound for the necessary deficiency at the same level of $T$.
Our results fall one step short of this precise characterization, which we leave as an open question.


\end{document}